\documentclass[11pt]{amsart}
\usepackage{amsmath,amssymb,etoolbox,mathrsfs,booktabs,float,enumerate}

\numberwithin{equation}{section}
\theoremstyle{plain}
\newtheorem{theorem}{Theorem}[section]
\newtheorem{proposition}[theorem]{Proposition}
\newtheorem*{proposition*}{Proposition}
\newtheorem{lemma}[theorem]{Lemma}
\newtheorem{lemmma*}{Lemma}

\theoremstyle{definition}
\newtheorem{definition}[theorem]{Definition}
\newtheorem*{definition*}{Definition}

\newtheorem{remark*}{Remark}

\newcommand{\Irr}{\mathrm{Irr}}

\newcommand{\des}{\mathrm{des}}
\newcommand{\CF}{\mathrm{CF}}

\newcommand\M{{\sf M}}

\newcommand\emptysq{[\,\text{-}\, , \,\text{-}\,]}

\newcommand\lcm{\mathrm{lcm}}
\newcommand{\ch}{\mathrm{ch}}

\renewcommand{\P}{\mathcal P}
\newcommand{\FLatt}{\mathcal X}
\newcommand{\LLatt}{\mathcal Y}
\newcommand{\CLatt}{\mathcal Z}

\newcommand{\FCone}{\mathcal F}
\newcommand{\LCone}{\mathcal L}
\renewcommand{\P}{\mathcal P}
\renewcommand{\a}{\mathfrak a}

\newcommand{\D}{\mathscr D}
\newcommand{\E}{\mathscr E}
\newcommand{\A}{\mathcal A}
\newcommand{\fc}{\phi}
\newcommand{\tc}{\gamma}
\newcommand{\cc}{\psi}
\newcommand{\lc}{\omega}
\newcommand{\ec}{\varepsilon}
\newcommand{\thc}{\theta}

\newcommand{\NN}{\mathbb N}
\newcommand{\ZZ}{\mathbb Z}
\newcommand{\QQ}{\mathbb Q}

\newcommand{\CC}{\mathbb C}
\begin{document}
\title[The characters of $S_n$ that depend only on length]{The characters of symmetric groups that depend only on length}
\author[Alexander Rossi Miller]{Alexander Rossi Miller}
\email{alexander.r.miller@univie.ac.at}
\begin{abstract}
  The characters $\chi(\pi)$ of $S_n$ that depend only on
  the number of cycles of $\pi$ are completely determined. 
\end{abstract}
\maketitle
\thispagestyle{empty}

\section{Introduction}
Let $\ell(\pi)$ denote the number of cycles of a permutation $\pi\in S_n$. 
Let $\fc_0,\fc_1,\ldots,\fc_{n-1}$ be the Foulkes characters of $S_n$,
so $\fc_i$ is afforded by the sum of Specht modules $V_\beta$ with 
$\beta$ of border shape with exactly $n$ boxes and $i+1$ rows.
For history, properties, and a number of recent 
developments, 
including
the Diaconis--Fulman connection with adding random numbers,
see \cite{DF1,DF2,KT,M1,M2,M3,M4,M5}.

The
Foulkes characters have
long been studied for their
remarkable properties:
\begin{enumerate}[\it a)\ ]
\item
  They depend only on length in the sense that
  \[
    \fc_i(\pi)=\fc_i(\sigma)\ \text{whenever}\ \ell(\pi)=\ell(\sigma).
  \]
\item
  They form a basis for the space $\CF_\ell(S_n)$ of
  class functions of $S_n$ that depend only on $\ell$,
  with each $\thc\in\CF_\ell(S_n)$
  decomposing uniquely~as
  \[
    \thc
    =\sum_{i=0}^{n-1} \frac{\langle \thc,\ec_i\rangle}{\ec_i(1)}\fc_i,
  \]
  where $\ec_i$ is the irreducible character $\chi_{(n-i,1,\ldots,1)}$.
\item
  They decompose the character $\rho$ of the regular representation:
  \[\fc_0+\fc_1+\ldots+\fc_{n-1}=\rho.\]
\item
  Their degrees are the Eulerian numbers:
  \[
    \fc_i(1)=\#\{\pi\in S_n \mid \des(\pi)=i\},\quad
    \des(\pi)=\#\{i \mid \pi(i)>\pi(i+1)\}.
  \]
\item
  They branch according to
  \[
    \fc_i|_{S_{n-1}}=(n-i)\fc_{i-1}+(i+1)\fc_i,\ \ \text{where\ \ $\fc_{-1}=0$}.
  \]
\item
  And they even admit an explicit expression:
  \[
    \fc_i(\pi)=\sum_{j=0}^{n-1}(-1)^{i-j}\binom{n+1}{i-j}(j+1)^{\ell(\pi)},
  \]
  with the usual convention, used throughout this paper, that $\binom{u}{v}=0$ if $v$ does not satisfy $u\geq v\geq 0$.
\end{enumerate}

Two missing properties were recently added to the list in \cite{M5}, 
the first being a solution to the problem of decomposing products:
\begin{enumerate}[\it g)\ ]
\item For any two Foulkes characters $\fc_i$ and $\fc_j$ of $S_n$, 
  \[
    \fc_i\fc_j=\sum_{k=0}^{n-1} c_{ijk}  \fc_k,
  \]
  where 
  \[
    c_{ijk}=
    \#\{(x,y)\in S_n\times S_n \mid \des(x)=i,\ \des(y)=j,\ xy=z\}
  \]
for any fixed $z\in S_n$ with $\des(z)=k$, or more explicitly, 
  \[
    c_{ijk}=\sum_{{0\leq u\leq i}\atop{0\leq v\leq j}}
    (-1)^{i-u+j-v}\binom{n+1}{i-u}\binom{n+1}{j-v}\binom{uv+u+v+n-k}{n}.
  \]
\end{enumerate}
In addition to the combinatorial solution and the explicit solution presented here, a recursive solution was also found by specializing results  
of Delsarte that predate the introduction of Foulkes characters and were discovered in a completely different context, see \cite{M5}.
The combinatorial solution 
follows from an earlier result \cite[Thm.\ 9]{M1}
that connects
Foulkes characters with Eulerian
idempotents from cyclic homology, namely that
\[
  \D_i = \sum_{j=1}^{n}\fc_i(C_{j})\E_{j-1},
\]
where $\fc_i(C_k)$ denotes the value $\fc_i(\sigma)$ at any 
$\sigma\in C_k=\{\pi\in S_n \mid \ell(\pi)=k\}$,
\[
  \D_i=\sum_{{\pi\in S_n}\atop {\des(\pi)=i}}\pi,
\]
and the $\E_i$'s are the Eulerian idempotents, which are certain orthogonal idempotents in $\QQ[S_n]$, see \cite[(2.5)]{M5}.

The second property added to the list in \cite{M5}
solves the problem of finding a natural description of
the unique inner product on $\CF_\ell(S_n)$ with respect to which the 
$\fc_i$'s form an orthonormal basis:

\begin{enumerate}[\it h)\ ]
\item
  The Foulkes characters $\fc_0,\fc_1,\ldots,\fc_{n-1}$ of $S_n$ form
  an orthonormal basis for the Hilbert space $\CF_\ell(S_n)$ with
  inner product $\emptysq$ defined by 
  \[[\thc,\psi ]=\frac{1}{|S_n|}\sum_{i,j=1}^n \thc(C_i)\overline{\psi(C_j)}\mathbf E |\sigma C_i\cap \tau C_j|,\]
  where $\sigma$ and $\tau$ are $n$-cycles chosen uniformly at random from $C_1$,
  so $\mathbf E |\sigma C_i\cap \tau C_j|$ equals the expected number of ways that the product $\sigma\tau$ 
  can be written as a product $\alpha\beta$ with $\alpha\in C_i$ and $\beta\in C_j$. 
\end{enumerate}
This inner product $\emptysq$ gives a new and natural construction of Foulkes characters by applying the Gram--Schmidt process to a natural choice of basis for $\CF_\ell(S_n)$ that is composed of characters, analogous to how the irreducible characters themselves can be constructed by applying the Gram--Schmidt process to the characters $(1_{S_\lambda})^{S_n}$, see \cite{M5}.

The author gave a different construction of Foulkes characters in \cite{M1} using 
sums of modules afforded by reduced homology groups of
subcomplexes of certain equivariant strong deformation retracts of
wedges of spheres called Milnor fibers coming from invariant theory, which
works not only for ${S_n\cong G(1,1,n)}$ but for all of the full monomial groups $G(r,1,n)$,  
including the hyperoctahedral group $G(2,1,n)$.
In this more general setting, the role of $\ell$ is played by $n-\mathfrak l$, where
$\mathfrak l$ is the most natural choice of ``length'',
\[\mathfrak l(x)=\min\{k\geq 0 \mid x=y_1y_2\ldots y_k\ \text{for some reflections }y_i\in G(r,1,n)\}.\]
In addition to new properties and proofs in the classical case,
analogues of all the properties that we have described so far have
been established for $G(r,1,n)$. These generalized Foulkes characters
also have connections with certain Markov chains, just as in the case of $S_n$.
Most notably, Diaconis and Fulman \cite{DF2}
connected the hyperoctahedral Foulkes characters 
with a Markov chain for adding random numbers
in balanced ternary, a number system that reduces carries and
that Donald Knuth famously described as 
 ``perhaps the prettiest number system of all.''

But what is missing 
from the list of properties presented so far is 
an analogue of the fundamental characterization of characters,
by which we mean if $\chi_0,\chi_1,\ldots,\chi_{k-1}$
are the irreducible characters of a finite group~$G$, then
the characters of $G$ are the non-negative integer linear combinations
\[\a_0\chi_0+\a_1\chi_1+\ldots+\a_{k-1}\chi_{k-1},\quad
  \a=(\a_0,\a_1,\ldots,\a_{k-1})^t \in\NN^k,\]
where by $\NN$ we mean the set of non-negative integers.
Interestingly, for the hyperoctahedral groups and all the other full monomial groups with $r\geq 2$, there is an analogue of this characterization~\cite{M2}:
\begin{enumerate}[\it{i\,$'$)}\ ]
\item 
If $r\geq 2$,
then the characters of $G(r,1,n)$ that depend only on length are precisely the
non-negative integer linear combinations of the Foulkes characters of the group.
\end{enumerate}
So for the monomial groups $G(r,1,n)$ with $r\geq 2$, the picture is complete.

For the symmetric group, however, 
the problem of determining the characters that depend only on $\ell$  is more complicated and
has remained open for a few decades. 
Already for $n=3$, one finds that not all characters in $\CF_\ell(S_n)$ are integer linear combinations of Foulkes characters, for example $\chi_{(2,1)}$, 
and it was recently shown in~\cite{M2}, using symmetric functions and Chebyshev's theorem about primes, that the same is true for each $n\geq 3$, but little else is known about the characters in $\CF_\ell(S_n)$.

Given this gap in the story for $S_n$, a natural question to ask~is, instead of the
Foulkes characters, does 
there exist a different sequence of $n$ characters that is better suited for the role of 
irreducibles in $\CF_\ell(S_n)$ 
 in that their non-negative integer linear combinations are
the characters that depend only on length?
Our first theorem answers in the negative.

\begin{theorem}\label{theorem no better}
  If $n>3$, then there does not exist a collection of $n$ characters 
  $\chi_0,\chi_1,\ldots,\chi_{n-1}\in\CF_\ell(S_n)$ such that every character of $S_n$ that
  depends only on $\ell$ can be written as a non-negative integer linear combination of the $\chi_i$'s. 
\end{theorem}

In our main theorem, we completely resolve the problem
of determining the characters of $S_n$ that depend only on length.
Our solution is explicit and complete with a parametrization
in terms of non-negative integers, just as in the case of ordinary
characters. 
Let $\{r\}=r-\lfloor r\rfloor$ denote the fractional part of any given
rational number $r$.

\begin{theorem}\label{characters}
  The characters of $S_n$ that depend only on $\ell$ are the linear combinations
  \[\thc_{\a}=\tilde{\a}_0\fc_0+\tilde{\a}_1\fc_1+\ldots+\tilde{\a}_{n-1}\fc_{n-1},\quad \mathfrak a\in\NN^n,\]
  \[\tilde{\a}_k=\left\lfloor \frac{\a_k}{d_{k+1}}\right\rfloor +\left\{\sum_{j=0}^{n-1}\binom{n-k-1}{j-k}\frac{\a_j}{d_{j+1}}\right\},\quad d_k=\frac{k}{\gcd(n,k)},\]
  and moreover,
  \[\thc_{\mathfrak a}=\thc_{\mathfrak b}\quad\text{if and only if}\quad
    \mathfrak a=\mathfrak b.\]
\end{theorem}

As a consequence, we find that the non-negative integer linear combinations of Foulkes characters of $S_n$ have exponentially decaying density among the characters of $S_n$ that depend only on length. We also obtain an upper bound for this density and information about the denominators of the rational coefficients that occur in Theorem \ref{characters}.

\begin{theorem}\label{piped count}
  The number of characters of $S_n$ that depend only on $\ell$ and lie
  in the fundamental parallelepiped
  $\{\sum t_i\fc_i \mid t_i\in [0,1)\}$ 
  equals
  \[\frac{n!}{\gcd(1,n)\gcd(2,n)\ldots\gcd(n,n)}.\]
\end{theorem}

\begin{theorem}\label{clear denom}
  If $\sigma_n$ denotes the smallest positive integer such that,
  for each character $\chi$ of $S_n$ that depends only on $\ell$, 
  $\sigma_n\chi$ is a non-negative integer linear combination of Foulkes characters, then
  \[\sigma_n=\frac{\lcm(1,2,\ldots,n)}{n}=\frac{e^{f(n)}}{n},\]
  where $f$ is the second Chebyshev function.
\end{theorem}

Asymptotics of $f$, and in turn $\sigma_n$, have important number-theoretic significance,
with the prime number theorem being equivalent to 
\[f(n)\sim n\]
and more precise statements being equivalent to the Riemann Hypothesis.

\section{Number theoretic preliminaries}
Important for our analysis is a certain result which compliments and improves 
an old result of Cauchy.

Given a partition $\lambda$ of $n$, abbreviated $\lambda\vdash n$, we shall write $m_k(\lambda)$ for the number of parts of $\lambda$ that equal $k$, 
\[\ell(\lambda)=\sum_{k=1}^n m_k(\lambda),\]
and
\[\M(\lambda)=\binom{\ell(\lambda)}{m_1(\lambda),m_2(\lambda),\ldots,m_n(\lambda)}=
  \frac{\ell(\lambda)!}{m_1(\lambda)!m_2(\lambda)!\ldots m_n(\lambda)!}.\]
Sch\"onemann proved that 
\[\frac{\gcd(m_1(\lambda),m_2(\lambda),\ldots,m_n(\lambda))}{\ell(\lambda)}\M(\lambda)\in\ZZ\]
and 
Cauchy independently proved  that 
\[\frac{n}{\ell(\lambda)}\M(\lambda)\in \ZZ,\]
see Dickson's book \cite[Chap.~IX, p.~265]{Dickson}.
We find the exact gcd of the
$\M(\lambda)$'s where $\lambda$ partitions $n$ into exactly $k$ non-zero parts. 

\begin{proposition}\label{gcd}
  For any positive integers $n\geq k\geq 1$, 
  \begin{equation}
    \gcd\{\M(\lambda) \mid \lambda\vdash n,\ \ell(\lambda)=k\}
    =\frac{k}{\gcd(n,k)}.
  \end{equation}
\end{proposition}

\begin{proof}
  Fixing positive integers $n$ and $k$ with $n\geq k$, let
  \[g=\gcd\{ \M(\lambda)\mid \lambda\vdash n,\ \ell(\lambda)=k\}.\]

  Let $\lambda$ be a partition of $n$ with
  $\ell(\lambda)=k$.
  For $1\leq j\leq n$,
  \[\frac{m_j(\lambda)}{k}\binom{k}{m_j(\lambda)}=\binom{k-1}{m_j(\lambda)-1},\]
  so
  \[\frac{m_j(\lambda)}{k}\M(\lambda)\in \ZZ.\]
  Multiplying the left-hand side by $j$ and then summing over $j$ gives Cauchy's result 
  \[\frac{n}{k}\M(\lambda)\in \ZZ.\]
  Writing $\gcd(n,k)=un+vk$ with $u,v\in\ZZ$, we therefore have 
  \[\frac{\gcd(n,k)}{k}\M(\lambda)=u\frac{n}{k}\M(\lambda)+v\frac{k}{k}\M(\lambda)\in\ZZ.\]
  Hence $\frac{k}{\gcd(n,k)}$ divides $g$.

  To show that $g$ divides $\frac{k}{\gcd(n,k)}$, we 
  show that for each prime $p$,
  \[v_p(g)\leq v_p\left(\frac{k}{\gcd(n,k)}\right),\]
  where $v_p(m)$ denotes the $p$-adic valuation of $m$. 
  Fix $p$ and let $e=v_p(k)$, so $p^e\mid k$ and $p^{e+1}\nmid k$. 
  Write $n=ak+r$ for non-negative integers $a$ and $r$ with $0\leq r<k$.
  Similarly write $r=bp^e+s$ with $0\leq s<p^e$.
  Writing $u=p^{v_p(s)}$, let $\mu$ be the partition with $u$ parts
  of size $a+\frac{r}{u}$ and $k-u$ parts of size $a$.
  Then $\mu$ is a partition of $n$, $\ell(\mu)=k$, and 
  \[\M(\mu)=\binom{k}{u}.\]
  Using Kummer's theorem, then 
  \[v_p(\M(\mu))=
    v_p(k)-v_p(u)=
    v_p\left(\frac{k}{\gcd(n,k)}\right).\]
  Hence $v_p(g)\leq v_p(\frac{k}{\gcd(n,k)})$.
\end{proof}

\section{Proof of Theorem \ref{theorem no better}}
\begin{lemma}\label{particular character}
  The Frobenius characteristic of $(n-1)^{\ell -1}\in\CF_\ell(S_n)$ is a
  non-negative integer linear combination of
  the symmetric functions $h_\lambda$.
  In particular, the class function $(n-1)^{\ell-1}$ is a character of $S_n$.
\end{lemma}

\begin{proof}
  The case $n=1$ is trivial, so assume $n>1$. By \cite[Cor.\ 8]{M2}, the Frobenius characteristic $\ch(\thc)$ of the class function  $\thc(\pi)=(n-1)^{\ell(\pi) -1}$ satisfies 
  \[\ch(\thc)=\frac{1}{n-1}\sum_{\lambda\vdash n}c_\lambda h_\lambda,\]
  \[c_\lambda=\binom{n-1}{\ell(\lambda)}\M(\lambda)=
    \binom{n-1}{m_1(\lambda),m_2(\lambda),\ldots,m_n(\lambda),n-\ell(\lambda)-1}.\]
  For any non-negative integers $k_1,k_2,\ldots,k_s$ with sum $k$, we have
  \[\frac{k_i}{k}\binom{k}{k_1,k_2,\ldots,k_s}=\binom{k-1}{k_1,k_2,\ldots,k_{i-1},k_i-1,k_{i+1},\ldots,k_s}\in\ZZ.\]
  So for $\lambda\vdash n$ and $1\leq k\leq n$,
  \[\frac{m_k(\lambda)}{n-1}c_\lambda\in\ZZ,\]
  and in turn
  \[\frac{n}{n-1}c_\lambda=\sum_{k=1}^n k\frac{m_k(\lambda)}{n-1}c_\lambda\in\ZZ.\]
  Therefore, for each partition $\lambda$ of $n$, 
  \[\frac{1}{n-1}c_\lambda=\frac{n}{n-1}c_\lambda-c_\lambda\in \ZZ.\]
  So $\ch(\thc)$ is a non-negative integer linear combination of the $h_\lambda$'s.
\end{proof}

\begin{proof}[\textbf{\textit{Proof of Theorem \ref{theorem no better}}}]
  Suppose that $n>3$ and $\chi_0,\chi_1,\ldots,\chi_{n-1}$ are characters in $\CF_\ell(S_n)$
  with the property that every character in $\CF_\ell(S_n)$ is a non-negative
  integer linear combination of the $\chi_i$'s.
  Then the change of basis matrices 
  $A=(A_{ij})_{0\leq i,j\leq n-1}$ and $A^{-1}$, where  
  $\chi_j=\sum_i A_{ij}\fc_i$,
  both have all entries non-negative. 
  If at least two entries in
  a row of $A$ were non-zero, say in columns $s$ and $t$,
  then from the non-negativity and the relation $AA^{-1}=I$,
  the $s$ and $t$ rows of $A^{-1}$ would be linearly
  dependent. So $A$ must be a positive
  definite diagonal matrix times a permutation matrix, meaning that   
  for some positive scalars $s_0,s_1,\ldots, s_{n-1}$ 
  and some permutation $\sigma$ of the indices,
  $\chi_j=s_j\fc_{\sigma(j)}$ for $0\leq j\leq n-1$.

  Let $\thc$ be the class function of $S_n$ given by
  \[\thc(\pi)=(n-1)^{\ell(\pi)-1}.\]
  By Lemma \ref{particular character}, $\thc$ is a character of $S_n$.
  We claim that $\thc$ can not be written as a non-negative integer linear combination of characters $\eta_0,\eta_1,\ldots,\eta_{n-1}$ with $\eta_i\in\{s\phi_i\mid s\geq 0\}$ for $0\leq i\leq n-1$. 
  Equivalently, if $\thc$ is written as 
  \begin{equation}\label{theta decomp}
    \thc=c_0\fc_0+c_1\fc_1+\ldots +c_{n-1}\fc_{n-1},
  \end{equation}
  then at least one of the summands $c_i\fc_i$ is not a character.
  We will show that in fact $c_{n-2}\fc_{n-2}$ is not a character.

  Let $\tc_0,\tc_1,\ldots,\tc_{n-1}$ be the characters of $S_n$ given by
  \[\tc_k(\pi)=(k+1)^{\ell(\pi)},\]
  so $\tc_k$ is afforded by $(\CC^{k+1})^{\otimes n}$ with
  \[\pi.(u_1\otimes u_2\otimes\ldots \otimes u_n)=
    u_{\pi^{-1}(1)}\otimes u_{\pi^{-1}(2)}\otimes \ldots \otimes u_{\pi^{-1}(n)}.\]
  Then 
  \[\thc=\frac{\tc_{n-2}}{n-1}\]
  and 
  \begin{equation}\label{phi as psi sum}
    \fc_i=\sum_{j=0}^{n-1} (-1)^{i-j}\binom{n+1}{i-j} \tc_j,\quad 0\leq i\leq n-1.
  \end{equation}
  The inverse of the matrix
  \[\left[ (-1)^{i-j}\binom{n+1}{i-j}\right]_{0\leq i,j\leq n-1}\]
  equals
  \[\left[\binom{n+i-j}{n}\right]_{0\leq i,j\leq n-1},\]
  so 
  \begin{equation}\label{psi as phi sum}
    \tc_i=\sum_{j=0}^{n-1} \binom{n+i-j}{n} \fc_j,\quad 0\leq i\leq n-1.
  \end{equation}
  Therefore, the coefficient of $\fc_j$ in $\thc$ is 
  \begin{equation}\label{coeff cj}
    c_j=\frac{1}{n-1}\binom{2n-2-j}{n},\quad 0\leq j\leq n-1.
  \end{equation}
  
  Let $\nu$ be the partition of $n$ with $m_2(\nu)=2$ and $m_1(\nu)=n-4$. Using that for any $k$,
  \[\langle k^\ell,\chi_\nu\rangle = \prod_{b}\frac{k+c(b)}{h(b)},\]
  where $b$ runs over the boxes in the diagram of $\nu$ and where $c(b)$ and $h(b)$ denote the content and hook length of $b$,
  we have
  \begin{equation}\label{gamma box}
    \langle \tc_{j},\chi_\nu\rangle = 
    \begin{cases}
      0  & \text{if $j\leq n-4$},\\
      \frac{(n-2)(n-3)}{2} & \text{if $j=n-3$},\\
      \frac{n(n-1)(n-3)}{2} & \text{if $j=n-2$},\\
      \frac{n^2(n+1)(n-3)}{4} & \text{if $j=n-1$}.\\
    \end{cases}
  \end{equation}
  Combining \eqref{phi as psi sum}, \eqref{coeff cj}, and \eqref{gamma box}, we have 
  \begin{align*}
    \langle c_{n-2}\fc_{n-2},\chi_\nu\rangle&=
    \frac{1}{n-1}
    \left[-\binom{n+1}{1}\frac{(n-2)(n-3)}{2}+\frac{n(n-1)(n-3)}{2}\right]\\
                                             &=\frac{n-3}{n-2},
  \end{align*}
  which is not an integer for $n>3$. 
  So $c_{n-2}\fc_{n-2}$ is not a character. 
\end{proof}

\section{Proofs of Theorems \ref{characters}--\ref{clear denom}}
Fixing $n$, we have a chain of lattices
\[\FLatt\subset \LLatt\subset \CLatt,\]
where
$\CLatt$ is the lattice of virtual characters of $S_n$, 
\[\CLatt =\left\{\textstyle{\sum_{\chi\in\Irr(S_n)}} a_\chi \chi \ \, \big|\, \ a_\chi\in\ZZ\right\},\]
$\LLatt$ is the lattice of virtual characters of $S_n$ that depend only on $\ell$, 
\[\LLatt =\{\chi\in\CLatt \mid \chi\ \text{depends only on $\ell$}\},\]
and $\FLatt$ is the sublattice of $\LLatt$ spanned by the Foulkes characters of $S_n$,  
  \[\FLatt=\left\{a_0 \fc_0+a_1\fc_1+\ldots+a_{n-1}\fc_{n-1}
      \mid
      a_i\in\ZZ\right\}.\]

  That the lattice of virtual characters that depend only on length can also be characterized as the set of virtual characters in the rational
  span of Foulkes characters is fundamental to our approach.
  \begin{proposition}\label{integral in rational span}
    \[\LLatt=\CLatt\cap\{r_0\fc_0+r_1\fc_1+\ldots+r_{n-1}\fc_{n-1}\mid r_i\in\QQ\}.\]
  \end{proposition}
  \begin{proof}
    If a rational linear combination of Foulkes characters is a virtual character,
    then it is a virtual character that depends only on $\ell$.
    For the other inclusion, if $\thc\in\LLatt$, then $\thc\in\CLatt$
    and
    \[\thc=\sum_{i=0}^{n-1}\frac{\langle \thc,\ec_i\rangle}{\ec_i(1)}\fc_i,\quad \ec_i=\chi_{(n-i,1,\ldots,1)},\]
   so $\thc$ is a virtual character in the rational span of Foulkes characters.
  \end{proof}

  Let $\P$ be the fundamental parallelepiped 
  \[\P=\{t_0\fc_0+t_1\fc_1+\ldots+t_{n-1}\fc_{n-1}
    \mid 
    t_i\in [0,1)\}.\]
  Then each $\thc\in \LLatt$
  decomposes uniquely as
  \[\thc=\thc_\FLatt+\thc_\P,\qquad \thc_{\FLatt}\in\FLatt,\quad \thc_{\P}\in\LLatt\cap \P.\]

  \subsection{}
  We are interested in the genuine characters in our lattices. Let 
  \[\LCone=\{\text{characters of $S_n$ that depend only on $\ell$}\}\]
  and
  \[\FCone=\{\text{non-negative integer linear combinations of $\fc_i$'s}\}.\]
  So $\FCone\subset\LCone$ and both $\FCone$ and $\LCone$ are closed under addition.

  \begin{lemma}\label{P intersections}
  \[\CLatt\cap\P=\LLatt\cap \P=\LCone\cap\P\subset \LCone.\]
  \end{lemma}

  \begin{proof}
    Let $\thc \in\CLatt\cap \P$. 
    Then $\thc$ is a linear combination of the Foulkes characters with non-negative coefficients, so $\thc$ depends only on length and satisfies $\langle\thc,\chi\rangle\geq 0$ for each $\chi\in\Irr(S_n)$. Since $\langle\thc,\chi\rangle$ is an integer for each $\chi\in\Irr(S_n)$, we conclude that $\thc$ is a character of $S_n$ that depends only on length. 
  \end{proof}

  \begin{proposition}\label{L cone decomp}
    The characters $\thc$ of $S_n$ that depend only on length are
    precisely the sums
    \[\thc=\thc_{\FCone}+\thc_\P,\qquad \thc_\FCone\in \FCone,\quad \thc_\P\in\CLatt\cap\P,\]
    and the components $\thc_{\FCone}$ and $\thc_\P$ are uniquely determined by $\thc$. Equivalently, the addition map $(x,y)\mapsto x+y$ takes
    $\FCone\times (\CLatt\cap \P)$ bijectively onto $\LCone$.
    Moreover, if
    \[\thc=\sum_{i=0}^{n-1} r_i\fc_i,\]
    then
    \[\thc_\FCone=\sum_{i=0}^{n-1} \lfloor r_i\rfloor\fc_i\quad\text{and}\quad
    \thc_\P=\sum_{i=0}^{n-1}\{r_i\}\fc_i.\]
  \end{proposition}

  \begin{proof}
    By Lemma \ref{P intersections}, we have a mapping 
    \[
      \alpha: \FCone\times (\CLatt\cap \P)\to \LCone\quad\text{given by}\quad
      \alpha(x,y)=x+y.
    \]

    Each element of $\LLatt$ is uniquely expressible as an element from the sublattice $\FLatt$ plus an element from the fundamental domain $\LLatt\cap \P=\CLatt\cap \P$, and we also have the
    inclusions  
    $\FCone\subset \FLatt$ and $\LCone\subset \LLatt$, so $\alpha$ is injective.

    To show that $\alpha$ is surjective, consider a character $\thc$ that depends only on length. Then
    \[\thc=\sum_{i=0}^{n-1} r_i\fc_i,\qquad
      r_i=\frac{\langle \thc,\chi_{(n-i,1,\ldots,1)}\rangle}{\binom{n-1}{i}}\geq 0. \]
    Hence
    \[\thc=\phi+\psi,\]
    where
    \[\phi=\sum_{i=0}^{n-1} \lfloor r_i\rfloor\fc_i\in \FCone\]
    and
    \[\psi=\thc-\phi=\sum_{i=0}^{n-1} \{r_i\}\fc_i\in \CLatt\cap\P.\]
  \end{proof}

  \subsection{}
We now proceed to describe a particularly good pair of bases for the lattice $\LLatt$ and the sublattice $\FLatt$, from which we will be able to better understand the fundamental domain $\CLatt\cap \P$ for $\FLatt\subset \LLatt$.
  Specifically, we find a basis for $\LLatt$ that extends to a basis
  of $\CLatt$ and has the property that certain multiples of the basis elements form a basis for the sublattice $\FLatt$.
  
  \subsubsection{}
  The new basis for $\FLatt$ will be denoted by $\cc_0,\cc_1,\ldots,\cc_{n-1}$.

  \begin{definition}
    Define $\cc_0,\cc_1,\ldots,\cc_{n-1}\in\CF_\ell(S_n)$  by
    \begin{equation}
      \cc_i=\sum_{j=0}^{n-1}(-1)^{i-j}\binom{i+1}{i-j}(j+1)^\ell.
    \end{equation}
  \end{definition}

  Denoting by $\tc_k$ the character $(k+1)^\ell$ in $\CF_\ell(S_n)$,
  we have the following relations
  between the $\cc_i$'s,
   the $\fc_i$'s, and the $\tc_i$'s.
   
   \begin{proposition}\label{relations}
    \begin{alignat}{2}
      \fc_i    &= \sum_{j=0}^{n-1}(-1)^{i-j}\binom{n+1}{i-j} \tc_j,&\qquad
    \tc_i  &= \sum_{j=0}^{n-1} \binom{n+i-j}{i-j} \fc_j,\label{fc gamma}\\
    \cc_i &=\sum_{j=0}^{n-1} (-1)^{i-j}\binom{i+1}{i-j} \tc_j,&\qquad
    \tc_i  &= \sum_{j=0}^{n-1} \binom{i+1}{i-j} \cc_j,\label{vp gamma}\\
    \fc_i    &= \sum_{j=0}^{n-1} (-1)^{i-j}\binom{n-j-1}{i-j}\cc_j,&\qquad
    \cc_i &= \sum_{j=0}^{n-1} \binom{n-j-1}{i-j} \fc_j.\label{fc vp}
  \end{alignat}
\end{proposition}

\begin{proof}
  The first equality in \eqref{fc gamma} is the usual well-known description of $\fc_i$. The second
  equality in \eqref{fc gamma} is \eqref{psi as phi sum},
  which can alternatively be obtained by first writing
  \[\tc_i
    =
    \sum_{j=0}^{n-1}\frac{\langle\tc_i,\varepsilon_j\rangle}{\binom{n-1}{j}}
    \fc_j,\]
  where $\varepsilon_j=\chi_{(n-j,1,\ldots,1)}$, and then using that for any partition $\lambda$ of $n$,
  \begin{equation}\label{gamma hook cont}
    \langle \tc_i,\chi_\lambda\rangle = \prod_{b} \frac{i+1+c(b)}{h(b)},
  \end{equation}
  where the product is taken over all boxes $b$ in the diagram of $\lambda$ and where $c(b)$ and $h(b)$ denote the content and hook length of $b$, 
  which gives
  \begin{equation}\label{gamma hook binom}
    \frac{\langle\tc_i,\varepsilon_j\rangle}{\binom{n-1}{j}}
    =\binom{n+i-j}{i-j}.
  \end{equation}

  The first equality in \eqref{vp gamma} is the definition of $\cc_i$, and the second equality  in \eqref{vp gamma} follows from the first by using the fact that
  the inverse of the matrix
  $[\binom{j}{i}]_{1\leq i,j\leq n}$ equals $[(-1)^{i-j}\binom{j}{i}]_{1\leq i,j\leq n}$. 

  For the second equality in \eqref{fc vp}, by combining
  the first relation in \eqref{vp gamma} and the second relation in \eqref{fc gamma}, we have
  \begin{equation}\label{first vp to fc}
    \cc_i=\sum_{j=0}^{n-1}\sum_{u=0}^{n-1}(-1)^{i-u}\binom{n+u-j}{u-j}\binom{i+1}{i-u}\fc_j.
  \end{equation}
  For $1\leq s,t\leq n$, 
  \[\sum_{u=1}^{n} \binom{n-s}{n-u}\binom{t}{u}=\binom{n-s+t}{n},\]
  so, for $1\leq s,t\leq n$, 
  \begin{equation}\label{general binom}
    \sum_{u=1}^{n} (-1)^{u-t}\binom{n-s+u}{n}\binom{t}{u}=\binom{n-s}{n-t}.
  \end{equation}
  Using \eqref{general binom} in \eqref{first vp to fc}, we conclude that 
  \[\cc_i=\sum_{j=0}^{n-1} \binom{n-j-1}{i-j}\fc_j.\]
  The first equality in \eqref{fc vp} follows from the second equality in
  \eqref{fc vp} by
  using both that the inverse of $[\binom{j}{i}]_{0\leq i,j\leq n-1}$ equals 
  $[(-1)^{i-j}\binom{j}{i}]_{0\leq i,j\leq n-1}$ and  that
  matrix inversion commutes with the operation of transposing across the
  anti-diagonal, i.e.
  $A\mapsto JA^tJ$ with $J=[\delta_{i+j,n-1}]_{0\leq i,j\leq n-1}$. 
\end{proof}

\begin{theorem}\label{bases FLatt}
  The sequences $\{\fc_i\}_{i=0}^{n-1},\ \{\tc_i\}_{i=0}^{n-1},\ \{\cc_i\}_{i=0}^{n-1}$ are bases for~$\FLatt$.
\end{theorem}

\begin{proof}
  The sequence of $\fc_i$'s is a basis for $\FLatt$, so 
  by Proposition~\ref{relations}, the sequence of $\tc_i$'s
  and the sequence of $\cc_i$'s are also bases for $\FLatt$.
\end{proof}

\begin{proposition}\label{ch vp res}
  The Frobenius characteristic $\ch(\cc_k)$ of the character 
  $\cc_k$ of $S_n$ satisfies 
    \begin{equation}\label{ch vp eq}
      \ch(\cc_k)
      =
      \sum_{{\lambda\vdash n}\atop \ell(\lambda)=k+1} \M(\lambda)h_\lambda.
    \end{equation}
  \end{proposition}
  \begin{proof}
    By \cite[Cor.\ 8]{M2}, 
    \begin{equation}
      \ch(\tc_j)=\sum_{\lambda\vdash n}\binom{j+1}{\ell(\lambda)}M(\lambda) h_\lambda,
    \end{equation}
    so
    \begin{align*}
      \ch(\cc_k) &=
                  \sum_{j=0}^{n-1}
                  (-1)^{k-j} \binom{k+1}{j+1} \sum_{\lambda\vdash n}\binom{j+1}{\ell(\lambda)} \M(\lambda) h_\lambda\\
                &=\sum_{\lambda\vdash n}
                  \M(\lambda) h_\lambda
                  \sum_{j=1}^n
                  (-1)^{k+1-j} \binom{k+1}{j}\binom{j}{\ell(\lambda)}\\
                &=\sum_{{\lambda\vdash n}\atop{\ell(\lambda)=k+1}} \M(\lambda) h_\lambda.
    \end{align*}
  \end{proof}

  \subsubsection{}
  Next, we show that certain scalar multiples of the $\cc_i$
  basis for $\FLatt$ give a basis for $\LLatt$ and that this
  basis for $\LLatt$ 
  can be extended to a basis for $\CLatt$.
  
    \begin{definition}\label{omega def}
    Define $\lc_0,\lc_1,\ldots,\lc_{n-1}\in\CF_\ell(S_n)$  by
    \begin{equation}
      \lc_k=\frac{1}{d_{k+1}}\cc_k,\quad d_k=\frac{k}{\gcd(n,k)}.
    \end{equation}
  \end{definition}

  \begin{proposition}\label{omega warm up}
    $\lc_0,\lc_1,\ldots,\lc_{n-1}$ are linearly independent characters in~$\LCone$. 
  \end{proposition}

  \begin{proof}
    The $\lc_i$'s are linearly independent because the $\cc_i$'s are
    linearly independent, and they are characters by Proposition~\ref{gcd} and Proposition~\ref{ch vp res}.
  \end{proof}
  
  \begin{theorem}\label{omega extend}
    $\lc_0,\lc_1,\ldots,\lc_{n-1}$
    can be extended to a basis of $\CLatt$.
  \end{theorem}
  
  \begin{proof}
    For $0\leq k \leq n-1$, let
    \[
      B_k=\{\ch^{-1}(h_\lambda)\mid \lambda\vdash n,\ \ell(\lambda)=k+1\},
    \] so $B_0, B_1,\ldots, B_{n-1}$ are pairwise
    disjoint and $B=\bigcup_{k=0}^{n-1} B_k$ is a basis for $\CLatt$.
    
    By Proposition~\ref{gcd} and Proposition~\ref{ch vp res}, for $0\leq k\leq n-1$,
    \begin{equation}\label{Bk sum}
      \lc_k=\sum_{\xi\in B_k} a_\xi \xi
    \end{equation}
    with coefficients $a_\xi$ that are positive integers satisfying
    \begin{equation}\label{Bk condition}
      \gcd\{a_\xi \mid \xi\in B_k\}=1.
    \end{equation}
        
    Now, if $S$ is any non-empty subset of $B$,
    and if $\chi$ is any character of shape 
    \[
      \chi=\sum_{\xi\in S} a_\xi \xi
    \]
    with positive integer coefficients $a_\xi$ that satisfy 
    $
      \gcd\{a_\xi\mid \xi\in S\}=1,
    $
    then
    \begin{equation}\label{S claim}
      \text{$\chi$ can be extended to a basis of $\textstyle{\bigoplus}_{\xi\in S}\ZZ\xi$,}
    \end{equation}
    which we establish by induction on $|S|$ as follows.
    The $|S|=1$ case is trivial. 
    Assuming the statement for all subsets of $B$ with exactly 
    $k$ elements, 
    suppose $\chi$ is of the described shape as a sum over a subset 
    $S\subset B$ with $|S|=k+1$.
    Let $T=S\smallsetminus\{\eta\}$ for some $\eta\in S$ and write 
    \[\chi=a\eta+b\theta,\]
    where $a$ is the coefficient of $\eta$ in $\chi$ and 
    $b=\gcd\{\text{coeff.\ of $\xi$ in $\chi$}\mid \xi\in T\}$,
    so $\gcd(a,b)=1$.
    Writing $au+bv=1$ for some $u,v\in \ZZ$,
    let $\psi=u\theta-v\eta$. Then
    \[
      \eta=u\chi - b\psi
      \quad\text{and}\quad
      \theta=v\chi+a\psi,
    \]
    so the pair $\chi,\psi$ is a basis for
    $\ZZ\eta\oplus \ZZ\theta$.
    By hypothesis, $\theta$ can be extended to a basis of
    $\bigoplus_{\xi\in T}\ZZ \xi$.
    Therefore, $\chi$ can be extended to a basis of
    $\bigoplus_{\xi\in S}\ZZ \xi$.

    By \eqref{Bk sum}, \eqref{Bk condition}, and \eqref{S claim}, 
    for $0\leq k\leq n-1$, 
    $\lc_k$
    can be extended to a basis of
    $\bigoplus_{\xi\in B_k} \ZZ \xi$. 
    So $\lc_0,\lc_1,\ldots,\lc_{n-1}$ can be extended to a basis
    of $\CLatt$.
  \end{proof}
  
  \begin{theorem}\label{omega basis LLatt}
 $\lc_0,\lc_1,\ldots,\lc_{n-1}$ is a basis for $\LLatt$.
  \end{theorem}

  \begin{proof}
    By Proposition \ref{omega warm up},
    the $\lc_i$'s are linearly independent, so it remains
    to show that they span $\LLatt$. 
    By Theorem \ref{bases FLatt},
    the $\fc_i$'s and the $\cc_i$'s are bases of $\FLatt$,
    so they have the same rational span, which, by Definition \ref{omega def}, 
    must be the rational span of the $\lc_i$'s.
    So by Proposition \ref{integral in rational span},
    \begin{equation}\label{LLatt as rat omega span}
      \LLatt
      =
      \CLatt\cap \{r_0\lc_0+r_1\lc_1+\ldots+r_{n-1}\lc_{n-1} \mid r_i\in\QQ\}.
    \end{equation}
    By Theorem \ref{omega extend},
    the right-hand side of \eqref{LLatt as rat omega span} is the set of
    integer linear combinations of $\lc_0,\lc_1,\ldots,\lc_{n-1}$. 
    So $\lc_0,\lc_1,\ldots,\lc_{n-1}$ is a basis of~$\LLatt$. 
\end{proof}

\subsection{Proofs of Theorems \ref{characters}--\ref{clear denom}}
From Theorem~\ref{bases FLatt} and Theorem~\ref{omega basis LLatt} we will obtain an explicit description of the characters of $S_n$ that lie 
in the fundamental parallelepiped~$\P$. We start by reading off
the number of these characters.

\begin{theorem}\label{count}
  \[\LLatt/\FLatt\cong \ZZ/d_1\ZZ\times \ZZ/d_2\ZZ\times\ldots\times \ZZ/d_n\ZZ,\quad d_k=\frac{k}{\gcd(n,k)}.\]
  In particular,
\[[\LLatt : \FLatt]=|\CLatt\cap\P|=d_1d_2\ldots d_n=\frac{n!}{\gcd(1,n)\gcd(2,n)\ldots\gcd(n,n)}.\]
\end{theorem}

\begin{proof}[Proof of Theorem \ref{count} and Theorem \ref{piped count}]
  By Theorem~\ref{bases FLatt}, Theorem~\ref{omega basis LLatt},
  and the relation $\cc_k=d_{k+1}\lc_k$ for $0\leq k\leq n-1$.
\end{proof}

\begin{proof}[Proof of Theorem~\ref{clear denom}]
  From Theorem \ref{count}, we have
  \[\sigma_n=\lcm(d_1,d_2,\ldots,d_n),\quad d_k=\frac{k}{\gcd(n,k)}.\]
  So
  \begin{align*}
    \sigma_n
       &=\frac{1}{n}\lcm\left(\frac{n}{\gcd(1,n)},\frac{2n}{\gcd(2,n)},
         \ldots,\frac{n^2}{\gcd(n,n)}\right)\\
       &=\frac{1}{n}\lcm(\lcm(1,n),\lcm(2,n),\ldots,\lcm(n,n))\\
       &=\frac{\lcm(1,2,\ldots,n)}{n}.
  \end{align*}
\end{proof}

\begin{theorem}\label{fund domain}
  The elements of $\CLatt\cap \P$ are
  \[\thc_{\a}=\tilde{\a}_0\fc_0+\tilde{\a}_1\fc_1+\ldots+\tilde{\a}_{n-1}\fc_{n-1},\quad
    \a\in \NN^n,\quad 0\leq \a_k<d_{k+1},\]
where 
\[\tilde{\a}_k=\left\{\sum_{j=0}^{n-1}\binom{n-k-1}{j-k}\frac{\a_j}{d_{j+1}}\right\},\quad d_k=\frac{k}{\gcd(n,k)}.\]
\end{theorem}

\begin{proof}
  Let
  \[\A=\{(\a_0,\a_1,\ldots,\a_{n-1})^t\in\NN^n \mid 0\leq \a_k<d_{k+1}\}.\]
  Then by Theorem~\ref{bases FLatt}, Theorem~\ref{omega basis LLatt},
  and the relation $\cc_k=d_{k+1}\lc_k$, 
  the characters
  \[\xi_\a=\a_0\lc_0+\a_1\lc_1+\ldots+\a_{n-1}\lc_{n-1},\quad \a\in\A,\]
  form a complete set of
  pairwise distinct representatives
  for the cosets in $\LLatt/\FLatt$.
  Using the second equality in \eqref{fc vp},
  \[\xi_\a=\hat{\a}_0\fc_0+\hat{\a}_1\fc_1+\ldots+\hat{\a}_{n-1}\fc_{n-1},\quad \a\in\A,\]
  where
  \[\hat{\a}_k=\sum_{j=0}^{n-1}\binom{n-k-1}{j-k}\frac{\a_j}{d_{j+1}}.\]
  So 
  \[\xi_\a+\FLatt=\thc_\a+\FLatt,\quad \a\in\A,\]
  and hence 
  the $\thc_\a$ with $\a\in\A$
  also form a complete set of pairwise distinct representatives
  for the cosets in $\LLatt/\FLatt$.
  Since $\thc_\a\in\P$ for each $\a\in\A$, we conclude
  that the $\thc_\a$'s
  are precisely the pairwise distinct elements of ${\CLatt\cap \P}$. 
\end{proof}

\begin{proof}[Proof of Theorem~\ref{characters}]
By Theorem \ref{fund domain} and Proposition \ref{L cone decomp}.
\end{proof}

\end{document}